\documentclass{amsart}
\usepackage{amssymb}
\usepackage{graphicx}
\usepackage{amscd}
\usepackage{stmaryrd}
\usepackage{color}

\allowdisplaybreaks

\newcommand{\ignore}[1]{}

\numberwithin{figure}{section}
\numberwithin{table}{section}

\newcommand\aint{{\int \hspace{- 10pt}- \hspace{- 5 pt}}}
\newcommand\tr{\operatorname{tr}}

\newcommand\grad{\operatorname{grad}}

\newcommand\supp{\operatorname{supp}}

\renewcommand\S{{\mathcal S}}

\newcommand\R{\mathbb{R}}
\newcommand\eps{\operatorname\epsilon}

\newcommand\B{{\mathcal B}}

\newcommand\I{{\mathcal I}}

\renewcommand\P{{\mathcal P}}

\newcommand\T{{\mathcal T}}

\newcommand{\0}{\mathaccent23}

\numberwithin{equation}{section}

\newtheorem{thm}{Theorem}[section]

\newtheorem{lem}[thm]{Lemma}
\newtheorem{cor}[thm]{Corollary}

\begin{document}

\title[The bubble transform]{The bubble transform: a new tool for
  analysis of finite element methods}
\author{Richard S. Falk}
\address{Department of Mathematics,
Rutgers University, Piscataway, NJ 08854}
\email{falk@math.rutgers.edu}
\urladdr{http://www.math.rutgers.edu/\char'176falk/}
\thanks{}
\author{Ragnar Winther}
\address{Department of Mathematics,
University of Oslo, 0316 Oslo, Norway}
\email{rwinther@math.uio.no}
\urladdr{http://heim.ifi.uio.no/\char'176rwinther/}
\thanks{}
\date{December 2, 2013}
\thanks{The work of the second author was supported by the Norwegian
Research Council.}

\begin{abstract}
The purpose of this paper is to discuss the construction of a linear
operator, referred to as the bubble transform, which maps scalar
functions defined on $\Omega \subset \R^n$ into a collection of
functions with local support. In fact, for a given simplicial
triangulation $\T$ of $\Omega$, the associated bubble transform
$\B_{\T}$ produces a decomposition of functions on $\Omega$ into a sum
of functions with support on the corresponding macroelements.  The
transform is bounded in both $L^2$ and the Sobolev space $H^1$, it is
local,
and it preserves the
corresponding continuous piecewise polynomial spaces.  As a
consequence, this transform is a useful tool for constructing local
projection operators into finite element spaces such that the appropriate
operator norms are bounded independently of polynomial degree. The
transform is basically constructed by two families of operators, local
averaging operators and rational trace preserving cut--off operators.
\end{abstract}

\maketitle
\section{Introduction}\label{intro}
Let $\Omega$ be a bounded polyhedral domain in $\R^n$ and $\T$
a simplicial triangulation of $\Omega$. The purpose of this paper is to
construct a decomposition of scalar functions on $\Omega $ into a sum
of functions with local support with respect to the triangulation
$\T$. The decomposition is defined by a linear map $\B
= \B_{\T}$, referred to as the bubble transform, which maps the
Sobolev space $H^1(\Omega)$ boundedly into a direct sum of local
spaces of the form $\0 H^1(\Omega_f)$, where $f$ runs over all the
subsimplexes of $\T$ and $\Omega_f$ denotes appropriate macroelements
associated to $f$. More precisely,
\[
\B = \sum_{f\in\Delta(\T)} B_f : H^1(\Omega) \to
     \sum_{f\in\Delta(\T)} \0H^1(\Omega_f),
\]
where the maps $\B_f : H^1(\Omega) \to \0H^1(\Omega_f)$
are local and bounded linear maps with the property that for all values
of $r\ge 1$, if $u$ is a continuous piecewise polynomial of degree
at most $r$ with respect to the triangulation $\T$, then $B_f u$ is a
continuous piecewise polynomial of degree at most $r$ with respect
to the restriction of the triangulation to $\Omega_f$.  Thus the map
$\B$ is independent of a particular polynomial degree $r$ and so does
not depend on a particular finite element space.

To motivate the construction of the bubble transform, let us recall that
the construction of projection operators is a key tool for deriving
stability results and convergence estimates for various finite element
methods. In particular, for the analysis of mixed finite element
methods, projection operators which commute with differential
operators have been a central feature since the beginning of such
analysis, cf. \cite{Brezzi, BrezziFortin}. Another setting where
such operators potentially would be very useful, but hard to
construct, is the analysis of the so-called
$p$-version of the finite element method, i.e., in the
setting where we are
interested in convergence properties as the polynomial degree of the
finite element spaces increases.
For such investigations,
the construction of projection operators which admit uniform bounds
with respect to polynomial degree represents a main challenge.
In fact, so far such constructions have appeared to be
substantially more difficult than the more standard analysis of the
finite element method, where the focus is on convergence with
respect to mesh refinement.

Pioneering results on the convergence of
the $p$-method applied to second order elliptic problems in two space
dimensions
were derived by Babu\v{s}ka and Suri \cite{MR899702}. An important
ingredient in their analysis was the construction of a
polynomial preserving extension operator. A generalization
of the construction to three space dimensions in the tetrahedral case
can be found in
\cite{MR1445738}, while the importance of such extension operators for
the Maxwell
equations was argued in \cite{MR2034876}.
Further developments of commuting extension
operators for the de Rham complex in three space dimensions are for
example presented
in \cite{MR2105164,MR2439500,MR2551195,MR2904580}.
These constructions have been used to establish a number of convergence
results
for the $p$-method, not only for boundary value problems, but also
for eigenvalue problems \cite{MR2764424}.
A crucial step in this analysis is the use of so--called projection
based interpolation operators,
cf. \cite[Chapter 3]{MR2459075} and \cite{MR2034876,MR2105164,hiptmair-2009}.
However, this development has not led to local projection operators
which are uniformly bounded in the appropriate Sobolev norms.
Some extra regularity seems to be necessary, cf. \cite[Section
6]{MR2764424}
or \cite[Section
4]{hiptmair-2009},
and, as a consequence,
the theory for the
$p$-method is far more technical than the corresponding theory for the
$h$-method.
This complexity represents a main obstacle for
generalizing the theory for the $p$-method
in various directions. The bubble transform introduced in this paper
represents a new tool which will be useful to overcome
some of these difficulties. In particular, the construction of
projection 
operators onto the
spaces of
continuous piecewise polynomials, which are uniformly bounded
in $H^1$ with respect to the polynomial degree,  is an immediate  consequence.

In practical computations,
improved accuracy is often achieved by
combining increased polynomial degree and mesh refinement, an
approach frequently referred to as the $hp$-finite element method.
However, for simplicity, throughout this paper we consider
the triangulation $\T$ to be fixed.
Although the discussion in this paper is restricted to scalar valued functions,
it will be convenient to use some of the notation defined for the more
general situation of the de Rham complex and differential forms
in \cite{acta,bulletin}.
In particular, we let $\Delta_j(\T)$ denote the set of subsimplexes
of dimension $j$ of the triangulation $\T$,
while
\[
\Delta(\T) = \bigcup_{j=0}^n  \Delta_j(\T)
\]
is the set of all subsimplexes.
Furthermore,
the space
$\P_r\Lambda^0(\T) \subset H^1(\Omega)$ is the space of continuous
piecewise polynomials of degree $r$ with respect to the triangulation
$\T$.
We recall that the spaces $\P_r\Lambda^0(\T)$ admit degrees of freedom
of the form
\begin{equation}\label{DOF}
\int_f u \, \eta, \quad \eta
\in \P_{r-1-\dim f}(f),\,  f \in \Delta(\T),
\end{equation}
where $\P_j(f)$ denotes the set of polynomials of degree $j$ on $f$.
These degrees of freedom
uniquely determine an element in $\P_r\Lambda^0(\T)$. In fact, the
degrees of freedom associated to a given simplex $f \in \Delta(\T)$
uniquely determine elements in $\0\P_r(f)$, the space of polynomials of
degree
$r$ on $f$ which vanish on the boundary
$\partial f$.

For each $f \in \Delta(\T)$, we let $\Omega_f$ be the macroelement
consisting of the union of the elements of $\T$ containing $f$, i.e.,
\[
\Omega_f = \bigcup \{T \, | \, T \in \T, \, f \in \Delta(T) \, \},
\]
while $\T_f$ is the restriction of the triangulation $\T$ to
$\Omega_f$.  It is a consequence of the properties of the degrees of
freedom that for each $f \in \Delta(\T)$, there exists an extension
operator $E_f : \0\P_r(f) \to \0\P_r\Lambda^0(\T_f)$. Here,
$\0\P_r\Lambda^0(\T_f)$ consists of all functions in
$\P_r\Lambda^0(\T_f)$ which are identically zero on $\Omega \setminus
\Omega_f$.  Furthermore, the space $\P_r\Lambda^0(\T)$ admits a direct
sum decomposition of the form
\begin{equation}\label{decomp1}
\P_r\Lambda^0(\T) = \bigoplus_{f \in
  \Delta(\T)}E_f(\0\P_r(f)) \subset \bigoplus_{f \in
  \Delta(\T)} \0\P_r\Lambda^0(\T_f).
\end{equation}
The extension operators $E_f$ introduced above, defined from the
degrees of freedom, will depend on the space $\P_r\Lambda^0(\T)$. In
particular, they depend on the polynomial degree $r$. However, it is a
key observation that the macroelements $\Omega_f$ only depend on the
triangulation $\T$, and not on $r$. So for all $r$, there exists a
decomposition of the space $\P_r\Lambda^0(\T)$ of the form
\eqref{decomp1}, i.e., into a direct sum of local spaces
$\0\P_r\Lambda^0(\T_f)$.  Furthermore, {\it the geometric structure}
of these decompositions, represented by the simplexes $f \in
\Delta(\T)$ and the associated macroelements $\Omega_f$, is
independent of $r$, and this indicates that a corresponding
decomposition may also exist for the space $H^1(\Omega)$ itself. More
precisely, the ansatz is a decomposition of $H^1(\Omega)$ of the form
$H^1(\Omega) = \bigoplus_f \0 H^1(\Omega_f)$.  The
bubble transform, $\B = \B_{\T}$, which we will introduce below,
produces such a decomposition. As noted above, the transform is a
bounded linear operator
\[
\B : H^1(\Omega) \to \bigoplus_{f \in
  \Delta(\T)} \0 H^1(\Omega_f)
\]
that preserves the piecewise
polynomial spaces of \eqref{decomp1} in the sense that if $u \in
\P_r\Lambda^0(\T)$, then each component of the transform, $B_fu$, is in
$\0\P_r\Lambda^0(\T_f) \subset \0 H^1(\Omega_f)$. In fact, $\B$ is also bounded in $L^2$.
The transform depends on the given
triangulation $\T$, but there is no finite element space present in
the construction.

We should note that once the transformation $\B$ is shown to exist,
the construction of local and uniformly bounded projections onto the spaces
$\P_r\Lambda^0(\T)$, with a bound independent of $r$,  is
  straightforward.
We just project each component  $B_fu \in \0 H^1(\Omega_f)$
by a local projection into the
subspace $\0\P_r\Lambda^0(\T_f)$. Since each local projection can be
chosen to have norm equal to one, the global operator mapping $u$ to
the local projections of $B_fu$ will be
bounded independently of the degree $r$. Furthermore, this process will
lead to a projection operator since the transform preserves continuous
piecewise
polynomials.

In fact, unisolvent degrees of freedom,
generalizing  \eqref{DOF},  exist for all the finite element spaces of
differential forms, referred to as $\P_r\Lambda^k(\T)$ and
$\P_r^-\Lambda^k(\T)$ and
studied in \cite{acta,bulletin}. As long as the triangulation $\T$ is
fixed, all these spaces admit degrees of freedom with a  common
geometric structure, independent of the polynomial degree $r$.
Therefore, for all these spaces there
exist degrees of freedom generalizing \eqref{DOF}, and local
decompositions similar to \eqref{decomp1}. So far these
decompositions have been utilized to derive basis functions in the
general setting, cf. \cite{decomp}, and to construct canonical, but
unbounded, local projections \cite[Section 5.2]{acta}. By combining
these canonical projections with appropriate smoothing operators,
bounded, but nonlocal projections which commute with the exterior
derivative
were also constructed in \cite{MR2373181,schoberl05} and 
\cite[Section 5.4]{acta}.
Furthermore, in \cite{local-cochain} local decompositions and a double
complex structure were the main tools to obtain
local and bounded
cochain projections
for the spaces $\P_r\Lambda^k(\T)$ and
$\P_r^-\Lambda^k(\T)$.
However, none of the projections just described
will admit bounds which are independent of the polynomial
degree
$r$, while the construction of projections with such
bounds is almost immediate from the properties of the bubble transform,
cf. Section~\ref{projections} below.
Therefore, it is our ambition to generalize the construction of
the bubble transform given below
to differential forms in any dimension, such that the transform
is bounded in the appropriate Sobolev norms, it commutes
with the exterior derivative, and it preserves the finite element
spaces $\P_r\Lambda^k(\T)$ and
$\P_r^-\Lambda^k(\T)$.
However, in the rest of this paper we restrict the discussion to
$0$-forms, i.e., to ordinary scalar valued functions defined on
$\Omega \subset \R^n$.

The present paper is organized as follows. In Section~\ref{prelim} we
present the main properties of the transform and introduce some useful
notation. The key tools needed for the construction are introduced in
Section~\ref{tools}. The main results of the paper are derived in
Section~\ref{main}. However, the verification of some of the more
technical estimates are delayed until Section~\ref{technical}.

\section{Preliminaries}\label{prelim}
We will use $H^1(\Omega)$ to denote the Sobolev space of all
functions  $L^2(\Omega)$ which also have the components of the
gradient in $L^2$, and $\| \cdot \|_1$ is the corresponding norm.
If $\Omega' \subset \Omega$, then $\| \cdot \|_{1,\Omega'}$
denotes the $H^1$ norm with respect to $\Omega'$.
The corresponding notation for the $L^2$-norms are $\| \cdot \|_0$
and $\| \cdot \|_{0,\Omega'}$.
Furthermore, if $\Omega_f$ is a macroelement
associated to $f \in \Delta(\T)$, then
\[
\0 H^1(\Omega_f) = \{ v \in H^1(\Omega_f)\, | \, \0 E_fv \in H^1(\Omega)
\, \},
\]
where $\0 E_f : L^2(\Omega_f) \to L^2(\Omega)$ denotes the the
extension by zero outside $\Omega_f$.
For any $f \in \Delta(\T)$,
$\Delta(f)$
is the set of subsimplexes of $f$.
In addition to the macroelements $\Omega_f$, we also introduce the
extended macroelements, $\Omega_f^e$, given by
\[
\Omega_f^e
  = \cup\{\Omega_g \, | \, g \in \Delta_0(\T) \, \}.
\]
It is a simple observation that if $g \in \Delta(f)$ then
$\Omega_g \supset \Omega_f$, while $\Omega_g^e \subset \Omega_f^e$.

\subsection{An overview of the construction}
The construction of the transformation $\B$ will be done inductively
with respect to the dimension of $f \in \Delta(\T)$.
We are
seeking a decomposition of the space  $H^1(\Omega)$
with properties similar to \eqref{decomp1}.
More precisely, we will establish that
any function $u \in H^1(\Omega)$ can be decomposed into
a sum, $u = \sum_f u_f$, where each component $u_f \in \0
H^1(\Omega_f)$. The map $u \mapsto u_f$ will be denoted $B_f$,
and the collection of all these maps can be seen as
a linear transformation $\B = \B_{\T} : H^1(\Omega)
\to \bigoplus_{f \in \Delta(\T)}\0 H^1(\Omega_f)$ with the following properties:
\begin{itemize}
\item[(i)] $u = \sum_f B_f u$, where the component map $B_f$
is a local operator mapping
  $H^1(\Omega_f^e)$ to $\0 H^1(\Omega_f)$.
\item[(ii)] $\B$ is bounded, i.e., there is a constant $c$, depending
  on the triangulation $\T$,  such that
\[
\sum_f \|B_f u \|_{1,\Omega_f}^2 \le c \|u \|_1^2, \quad u \in H^1(\Omega).
\]
\item[(iii)] $\B$ preserves the piecewise polynomial spaces in the sense that
\[
u \in \P_r\Lambda^0(\T) \Longrightarrow B_fu \in \0\P_r\Lambda^0(\T_f).
\]
\end{itemize}
In the special case when $n=1$ and $\Omega$ is an interval, say $\Omega =
(0,1)$, a transform with the properties above is easy to construct. In this
case, $\T$ is simply a partition of the form
\[
0= x_0 < x_1< \ldots <x_N = 1.
\]
The set $\Delta_0(\T)$ is the set of vertices $\{x_j\}$, while
$\Delta_1(\T)$ is the set of intervals of the form $(x_{j-1},x_j)$.
If $f = x_j \in \Delta_0(\T)$, then $\Omega_f = (x_{j-1},x_{j+1})$,
with an obvious modification near the boundary, while
$\Omega_f = f$ for $f \in \Delta_1(\T)$.
Let $\lambda_i \in \P_1\Lambda^0(\T)$ be the  standard piecewise linear ``hat
functions,'' characterized by $\lambda_i(x_j) = \delta_{i,j}$. For all $f
= x_j\in \Delta_0(\T)$, we let $B_f u = u(x_j)\lambda_j$. By construction,
$B_f u $ has support in $\Omega_f$. Furthermore, the function
\[
u^1 = u - \sum_{f \in \Delta_0(\T)} B_fu
\]
vanishes at all the vertices $x_j$. Therefore, if we let $B_f u =
u^1|_f$
for all $f \in \Delta_1(\T)$, then $B_f u \in \0 H^1(\Omega_f)$, and
$u = \sum_{f \in \Delta(\T)} B_fu$. In fact, it is straightforward to
check that all the properties (i)--(iii) hold for this construction.

In general, for $n>1$,  $\tr_f u$, for $f \in \Delta(\T)$,
will not be well defined
for $u \in H^1(\Omega)$. Therefore, the simple construction above
cannot be directly generalized to higher dimensions.
For example, when $f$ is the vertex $x_0$, to define $B_f u$,
we introduce the average of $u$
\[
U(x) = \frac{1}{|\Omega_f|} \int_{\Omega_f} u(\lambda_0(x) + [1 - \lambda_0(x)]
y) \, dy,
\]
where $\lambda_0(x)$ is now the $n$-dimensional piecewise linear function
equal to one at $x_0$ and zero at all other vertices.
Note that if $u$ is well-defined at $x_0$, then $U(x_0) = u(x_0)$,
while if $x\in \Omega\setminus\Omega_f$, then
$U(x)$ is just the average of $u$ over $\Omega_f$.  In general, for
$x \neq x_0$, $U(x)$ has pointwise values. Note that $U(x)$ depends only
on $\lambda_0(x)$, so is constant on level sets of $\lambda_0(x)$.

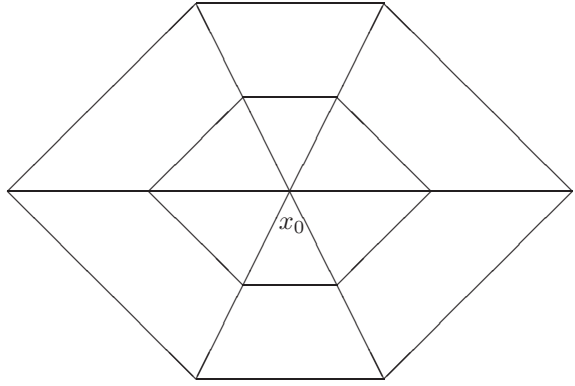
\begin{figure}[htb]
\setlength{\unitlength}{0.5cm}
\centering
\begin{picture}(15,6)
\put(0,0){\line(1,0){15}}
\put(0,0){\line(1,1){5}}
\put(0,0){\line(1,-1){5}}
\put(5,5){\line(1,0){5}}
\put(5,-5){\line(1,0){5}}
\put(15,0){\line(-1,1){5}}
\put(15,0){\line(-1,-1){5}}
\put(7.5,0){\line(-1,2){2.5}}
\put(7.5,0){\line(1,2){2.5}}
\put(7.5,0){\line(-1,-2){2.5}}
\put(7.5,0){\line(1,-2){2.5}}
\put(3.75,0){\line(1,1){2.5}}
\put(6.25,2.5){\line(1,0){2.5}}
\put(3.75,0){\line(1,-1){2.5}}
\put(6.25,-2.5){\line(1,0){2.5}}
\put(7.2,-1){$x_0$}
\put(11.25,0){\line(-1,1){2.5}}
\put(11.25,0){\line(-1,-1){2.5}}
\end{picture}
\vskip1in
\caption{The level set $\lambda_0(x)=1/2$ in the macroelement $\Omega_{x_0}$.}
\label{fig:levelset}
\end{figure}

In fact, if we replace $\lambda_0(x)$ by a variable $\lambda$ taking
values in $[0,1]$ in the definition of $U(x)$ above, then we may view
$U$ as a function of $\lambda$, which we will call $(A_f u)(\lambda)$.
Hence, $(A_f u)(\lambda_0(x)) = U(x)$.
It is easy to check that if $u$ is
a piecewise polynomial in $x$, then $A_f u$ is a polynomial in $\lambda$.
Finally, if we define
\begin{equation}\label{Bdef-prelim}
(B_f u)(x) = (A_f u)(\lambda_0(x)) - [1 - \lambda_0(x)] (A_f u)(0),
\end{equation}
then $B_f u$ will have support on $\Omega_f$.

For simplices $f$ of higher dimension, the operators $B_f$ will be
constructed recursively by a process of the form
\[
B_fu = C_f(u - \sum_{\stackrel{g \in \Delta(\T)}{\dim g < \dim
    f}} B_gu),
\]
where $C_f$ is a local trace preserving cut--off operator, i.e., designed
such that $C_fv$ is close to $v$ near $f$, but at the same time $C_fv$
vanishes outside $\Omega_f$. To also have $C_fv$ in $H^1$ will in general
require compatibility
conditions of $v$ on $\partial f \subset \partial \Omega_f$.
We will return to the precise definition of the operators $B_f$ and
$C_f$
in Section~\ref{main} below.


\subsection{Barycentric coordinates}
If $x_j \in \Delta_0(\T)$ is a vertex, then $\lambda_j(x) \in \P_1(\T)$ is
the corresponding barycentric coordinate, extended by zero outside
the corresponding macroelement. If $f \in \Delta_m(\T)$ has vertices
$x_0, x_1, \ldots, x_m$, then we write $[x_0,x_1, \ldots ,x_m]$ to
denote convex combinations, i.e.,
\[
f= [x_0,x_1, \ldots ,x_m] = \{\,  x = \sum_{j=0}^m \alpha_jx_j \, | \,
\sum_j \alpha_j =1, \, \alpha_j \ge 0 \, \}.
\]
The corresponding vector field $(\lambda_0, \lambda_1, \ldots ,
\lambda_m)$ with values in $\R^{m+1}$
is denoted $\lambda_f$.
Hence, the map $x \mapsto \lambda_f(x)$, restricted to $f$,
is a one-one map of $f$ onto $\S_m$, where
\[
\S_m = \{\, \lambda = (\lambda_0, \ldots ,\lambda_m) \in \R^{m+1} \, | \,
\sum_{j=0}^m \lambda_j = 1, \, \lambda_j \ge 0 \, \}.
\]
To the simplex $\S_m$ we associate the
simplex $\S_m^c = [\S_m,0]$,  given by
\[
\S_m^c = \{\, \lambda = (\lambda_0, \ldots ,\lambda_m) \in \R^{m+1} \, | \,
\sum_{j=0}^m \lambda_j \le 1, \, \lambda_j \ge 0 \, \}.
\]
Hence, $\S_m$ is an $m$ dimensional subsimplex of $\S_m^c$.
For $\lambda \in \S_m^c$, we define
\[
b(\lambda) = b_m(\lambda) = 1 - \sum_{j=0}^m \lambda_j,
\]
i.e., corresponding to the barycentric coordinate of the origin.

If $f = [x_0, x_1, \ldots ,x_m] \in \Delta_m(\T)$, then the
macroelements $\Omega_f$ and $\Omega_f^e$ are given by
\[
\Omega_f  = \bigcap_{j= 0}^{m} \Omega_{x_j} \quad \text{and }
\Omega_f^e = \bigcup_{j= 0}^{m} \Omega_{x_j}.
\]
The map
$x \mapsto \lambda_f(x)$ maps $\Omega$ to $\S_m^c$, $f$ to $\S_m$,
and the boundary $\partial \Omega_f$ to $\partial \S_m^c \setminus \S_m$,
cf. Figure~\ref{fig:map}.

\begin{figure}[htb]
\setlength{\unitlength}{0.35cm}
\centering
\begin{picture}(30,15)
\linethickness{2pt}
\put(5,0){\line(0,1){10}}
\linethickness{1pt}
\put(5,0){\line(2,1){10}}
\put(5,10){\line(2,-1){10}}
\put(5,0){\line(-3,4){5.15}}
\put(5,10){\line(-3,-2){5.15}}
\put(20,0){\vector(0,1){11}}
\put(20,0){\vector(1,0){11}}
\put(31.5,0){$\lambda_0$}
\put(19.8,11.5){$\lambda_1$}
\linethickness{2pt}
\qbezier(20,10)(25,5)(30,0)
\linethickness{1pt}
\put(7,3.5){$x$}
\put(5.1,5){$f$}
\put(4.8,-.6){$x_0$}
\put(4.8,10.5){$x_1$}
\put(7.5,10){$\Omega_f$}
\put(23,3.5){$\lambda_f(x)$}
\put(22.5,8){$\S_1$}
\put(28,10){$\S_1^C$}
\qbezier(7,3)(15,9)(23,3)
\put(22.4,3.7){\vector(1,-1){1}}
\end{picture}
\caption{The map $x \mapsto \lambda_f(x)$ for $n=2$ and $m=1$}
\label{fig:map}
\end{figure}
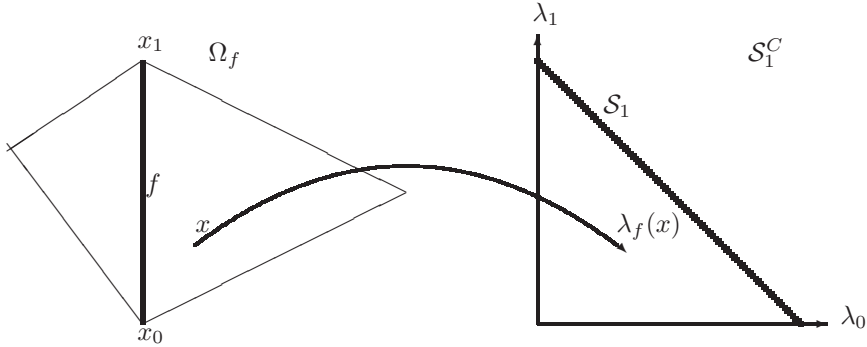

In particular, $\Omega \setminus \Omega_f^e$
is mapped to the origin.
For each $f = [x_0, x_1, \ldots ,x_m] \in \Delta_m(\T)$
we also introduce the piecewise linear function $\rho_f$ on $\Omega$ by
\[
\rho_f(x)= 1 - \sum_{j=0}^m \lambda_j(x) = b(\lambda_f(x)).
\]
As a consequence, the simplex $f$ can be characterized as
the null set of $\rho_f$, while $\rho_f \equiv 1$ on $\Omega \setminus
\Omega_f^e$.

For each integer $m \ge 0$, we let $\I_m$ be the
set of all subindexes of $(0,1,\ldots ,m)$, i.e., $\I_m$
corresponds to all subsets of  $\{0,1,\ldots ,m \}$.
In particular, we count the
empty set as an element of $\I_m$, such that $\I_m$ is a finite set with
$2^{m+1}$
elements. We will use $|I|$ to denote the cardinality of $I$.
If $ 0 \le i \le m$ is an integer,
then there are exactly $2^{m}$ elements of $\I_m$ which contain
$i$, and $2^{m}$ elements which do not contain
$i$. For any $I \in \I_m$, we define $P_I : \S_m^c \to \S_m^c$ by
\[
(P_I\lambda)_i =
\left\{ \begin{array}{ll} 0, \quad i \in I,\\
      \lambda_i, \quad i \notin I. \end{array}
\right.
\]
Hence if $I$ is nonempty, then $P_I$ maps the simplex $\S_m^c$ to
a portion of its boundary.
In particular, if $I =\{0,1,\ldots ,m \}$, then  $P_I$ maps $\S_m^c$ into
the origin of $\R^{m+1}$,
while $P_I$ is the identity if $I$ is the empty set.
Finally, for any $f \in \Delta_m(\T)$ and $I \in \I_m$ we let $f(I) \in
\Delta(f)$ denote the corresponding subsimplex of $f$ given by
$f(I) = \{ x \in f \, | \, P_I\lambda_f(x) = \lambda_f(x) \, \}.$
Hence, if $I$ is the empty set, then $f(I) = f$, while $f(I)$ is
the the empty subsimplex of $f$ if $I = (0,1, \ldots ,m) \in \I_m$.

\section{Tools for the construction}\label{tools}
The key tools for the construction are two families of operators,
referred to as trace preserving cut--off operators and
local averaging operators.

\subsection{The trace preserving cut off operator on $\S_m^c$}\label{trace-preserve}
Let $w$ be a real valued function defined on $\S_m^c$. For the
discussion in this section, we will assume that $w$ is
sufficiently regular to justify the operations below in a pointwise
sense.  We will introduce an operator $K= K_m$ which maps such
functions $w$ into a new function on $\S_m^c$, with the property that
the trace on $\S_m$ is preserved, but such that the trace of $K_mw$
vanishes on the rest of the boundary of $\S_m^c$.  In fact, the
operator $K_m$ resembles the extension operators discussed in
\cite{MR2439500}. However, in the present setting, where we will be
working with functions which may not have a trace on $\S_m$, trace
preserving operators seem to be a more useful concept.  The operator
$K_m$ can be viewed as a sum of pullbacks, weighted by rational
coefficients.  However, the operator $K_m$ preserves polynomials in an
appropriate sense, cf. Lemma~\ref{pol-preserve-E} below.  The operator
$K_m$ is defined by
\[
K_mw(\lambda) = \sum_{I \in \I_m}(-1)^{|I|}K_m^Iw
= \sum_{I \in \I_m}(-1)^{|I|}\frac{b(\lambda)}
{b(P_I\lambda)}w(P_I\lambda), \quad \lambda \in \S_m^c.
\]
When $m=0$, the set $\I_0$ has only two elements, the empty set and
$(0)$.
Therefore, the function $K_0$ maps functions $w= w(\lambda)$, defined
on $\S_0^c = [0,1]$,
to
\[
K_0w(\lambda) = w(\lambda) - (1- \lambda)w(0),
\]
such that \eqref{Bdef-prelim} can be rewritten as $B_fu = (K_0 \circ A_f)u(\lambda_0(\cdot))$.
We observe that $K_0w(1) = w(1)$, $K_0w(0) = 0$, and if $w \in \P_r$ then $K_0w
\in \P_r$.
Formally, we can also argue that $\tr_{\S_m}(w -K_mw) = 0$ for $m$
greater than zero.
This just follows since all the terms in the sum defining $K_m$, except
for the one corresponding to $I = \emptyset$,  i.e., $I$ is the
emptyset, have vanishing trace on $\S_m$ due to the appearance of the term
$b(\lambda)$ in the nominator.
A corresponding argument also shows that the trace of $K_mw$
vanishes on the rest of the boundary of $\S_m^c$. Recall that the
boundary of $\S_m^c$ consists of $\S_m$ and the subsimplexes
\[
\S_{m,i}= \{ \lambda \in \S_m^c \, |\, \lambda_i = 0 \, \} \quad
i=0,1, \ldots ,m.
\]
Furthermore, for a fixed $i$, let $I \in \I_m$ be any index such that
$i \notin I$, and let $I' \in \I_m$ be given as $I'= I \cup \{i
\}$. For
$\lambda \in \S_{m,i}$ we have $P_{I'}\lambda = P_I\lambda$, and therefore
\[
K_m^Iw(\lambda) - K_m^{I'}w(\lambda) = \frac{b(\lambda)}
{b(P_I\lambda)}w(P_I\lambda)
- \frac{b(\lambda) }{b(P_{I'}\lambda)}w(P_{I'}\lambda)= 0.
\]
However, for a fixed $i$ the set $\I_m$ is exactly equal to the union
of indexes of the form $I$ and $I'$. As a consequence, we conclude
that $K_mw$ is identically zero on $\S_{m,i}$, and hence on $\partial
\S_m^c \setminus \S_m$. In particular, $K_mw$ is zero at the origin.

The operator $K_m$ preserves  polynomials in the following sense.

\begin{lem}\label{pol-preserve-E}
Assume that $w \in \P_r(\S_m^c)$ with $\tr_{\S_m} w \in \0\P_r(\S_m)$.
Then $K_m w \in \P_r(\S_m^c)$, $\tr_{\S_m}(K_mw -w) = 0$, and
$\tr_{\partial \S_m^c\setminus \S_m} K_mw = 0$.
\end{lem}

\begin{proof}
Assume that $w \in \P_r(\S_m^c)$, such that $\tr_{\S_m}w$ vanishes on the
boundary of $\S_m$.  To show that  $K_m w \in \P_r(\S_m^c)$, we
consider each term in the sum defining $K_mw$ of the form
\[
K_m^Iw(\lambda) := \frac{b(\lambda)}
{b(P_I\lambda)}w(P_I\lambda).
\]
If $I = \emptyset$, then $K_m^Iw = w$, while if $I$ is the maximum set,
$I = (0,1,\ldots ,m)$,
then $K_m^Iw(\lambda) = b(\lambda)w(0,\ldots ,0)$ which is linear.
Therefore, it is enough to consider the other choices of $I$,
i.e., when $K_m^Iw$ has an essential rational coefficient
$b(\lambda)/b(P_I\lambda)$.

Note that since $\tr_{\S_m}w$ vanishes on the
boundary of $\S_m$,
we can conclude that $w(P_I\lambda)$ vanishes on
$\{\lambda \in \S_m^c\, |\, b(P_I \lambda) = 0 \, \}$.
This means that $w(P_I\lambda)$ must be of the form
$w(P_I\lambda) = b(P_I\lambda)w'(P_I\lambda)$,
where $w' \in \P_{r-1}(\S_{m,I})$. Here
\[
\S_{m,I} = \{ \lambda \in \S_m^c \, |\, P_I\lambda = \lambda \, \}.
\]
As a consequence,
$K_m^Iw = b(\lambda)w'(P_I\lambda ) \in \P_r(\S_m^c)$.
Furthermore, $\tr_{\S_m}K_mw = \tr_{\S_m} w$ since
all the terms $K_m^Iw$ have vanishing trace on $\S_m$, except for
the one corresponding to $I = \emptyset$. Finally,
the property that the trace of $K_mw$ vanishes on the rest of the
boundary of $\S_m^c$ follows from the discussion
given above.
\end{proof}

\subsection{The local averaging operator}
Throughout this section we will assume that
$f = [x_0,x_1, \ldots ,x_m] \in \Delta_m(\T)$, where
we assume that $0 \le m < n$.
For $v \in L^2(\Omega_f)$
and $\lambda \in \S_m^c$, we let  $A_fv(\lambda)$
be given by
\[
A_fv(\lambda) = \aint_{\Omega_f} v(y + \sum_{j=0}^m \lambda_j(x_j -y))\, dy,
\]
where the slash through an integral means an average, i.e.,
$\int_{\Omega_f} \hspace{- 18pt}- \quad$ should be interpreted as
$|\Omega_f|^{-1}\int_{\Omega_f}$.
If $\lambda \in \S_m$, then the integrand is independent of
$y$,
and therefore $A_fv(\lambda) = v(x)$, where $x = \sum_j \lambda_jx_j
\in f$. Hence, at least formally, the operator $\lambda_f^* \circ
A_f$, which is given by
$v \mapsto A_fv(\lambda_f(\cdot))$. is the identity operator on $f$.
We will find it convenient to introduce the function
$G = G_m: \S_m^c\times \Omega_f \to \Omega_f$ given by
\[
G_m(\lambda,y) = y + \sum_{j=0}^m \lambda_j(x_j -y) = \sum_{j=0}^m
\lambda_j x_j + b(\lambda)y, \quad
\lambda \in \S_m^c, \, y \in \Omega_f,
\]
so that the operator $A_f$ can be expressed as
\[
A_fv(\lambda) = \aint_{\Omega_f} v(G_m(\lambda,y))\, dy =
|\Omega_f|^{-1}\sum_{T \in \T_f}\int_{T} v(G_m(\lambda,y))\, dy.
\]
In fact, we observe that for each $y \in \Omega_f$, the map
$G_m(\cdot, y)$ maps $\S_m^c$ to $\Omega_f$, and the operator $A_f$
is simply the average with respect to $y$ of the pullbacks with
respect to these maps.
It is a property of the map $G_m$ that if $y \in T$, where $ T \in
\T_f$, then $G_m(\lambda,y) \in T$. In fact, $G_m(\lambda,y)$ is a convex combination of
$y$ and $b(\lambda)^{-1}\sum_i\lambda_ix_i \in f$.

A key property of the operator $A_f$ is that it maps
the piecewise polynomial spaces $\P_r\Lambda^0(\T_f)$ into
the polynomial spaces $\P_r(\S_m^c)$.

\begin{lem}\label{pol-preserve-A}
If $v \in \P_r\Lambda^0(\T)$, then $A_f v \in \P_r(\S_m^c)$.
Furthermore,  if $\lambda \in \S_m$,
then $A_f v(\lambda) = v(x)$, where $x = \sum_{j=0}^m \lambda_jx_j
\in f$.
\end{lem}

\begin{proof}
If $v \in \P_r\Lambda^0(\T)$, then
the restriction of $v$ to each triangle in $\T_f$ is a polynomial of
degree $r$.
Furthermore, the map $y \mapsto G_m(\lambda, y)$ maps each $T$ to
itself, and depends linearly on $\lambda$. Therefore,
$v(G_m(\lambda,y)) \in \P_r(\S_m^c)$ for each fixed $y$.
Taking the average over $\Omega_f$ with respect to $y$ preserves this
property, so $A_f v \in \P_r(\S_m^c)$.
The second result follows from the fact that the integrand is
independent of $y$, and equal to $v(\sum_j \lambda_jx_j)$, for
$\lambda \in \S_m $.
\end{proof}
We will also need mapping properties of the operator $\lambda_f^*
\circ A_f$. Since $\lambda_f$ maps all of $\Omega$ into $\S_m^c$, the
operator $\lambda_f^*\circ A_f$ maps a function
$v$ defined on $L^2(\Omega_f)$ to $A_fv(\lambda_f(\cdot))$ defined
on all of $\Omega$. It is a key result that this operator is bounded in $L^2$
and $H^1$. In fact, we even have the following.

\begin{lem}\label{H1-preserve-gen}
Assume that $f \in \Delta_m(\T)$ and $I \in \I_m$, with $m<n$.
The operator $\lambda_f^* \circ P_I^* \circ A_f$ is
bounded as an operator from $L^2(\Omega_f)$ to  $L^2(\Omega)$,
as well as from $H^1(\Omega_f)$ to $H^1(\Omega)$.
\end{lem}
The arguments involved to establish these boundedness results are
slightly more technical than the discussion
above. Therefore, we will delay the proof of this lemma, and the proofs of the next
three results below, to the final section of the paper.

As we have observed above, the operator $\lambda_f^* \circ A_f$
formally preserves traces on $f$. A weak formulation of this result
is expressed in the next lemma.

\begin{lem}\label{decay-create}
Assume that $f \in \Delta_m(\T)$ with $m <n$. Then
\[
\int_{\Omega} \rho_f^{-2}(x)|v(x) - A_fv(\lambda_f(x))|^2 \, dx
\le c \|v \|_1^2, \quad v \in H^1(\Omega),
\]
where the constant $c = c(\Omega,\T)$ is independent of $v$.
\end{lem}

Since the function $\rho_f(x)$ is identically zero on
$f$, this result shows that for any $v \in H^1(\Omega_f)$ ``the
error,'' $v - A_fv$,
has a decay
property near $f$.

The next result shows that the operator $\lambda_f^* \circ P_I^* \circ
A_f$ preserves such decay properties.

\begin{lem}\label{decay-preserve}
Assume that $f \in \Delta_m(\T)$ and $I \in \I_m$, with $m<n$,
and let $g = f(I) \in \Delta(f)$.
There is a constant $c=c(\Omega,\T)$, independent of $v$, such that
\[
\int_{\Omega}\rho_g^{-2}(x) |A_fv(P_I\lambda_f(x))|^2 \, dx \le c\,
\Big[\int_{\Omega}\rho_g^{-2}(x) |v(x)|^2 \, dx + \| \grad v \|_{0}^2\Big]
\]
for all $v \in H^1(\Omega)$, such that
$\rho_g^{-1} v \in L^2(\Omega)$.
\end{lem}

Finally, the following lemma will be a key ingredient
in the proof of Lemma~\ref{prop-2} to follow.

\begin{lem}\label{decay-cancel}
Assume that $f = [x_0,x_1, \ldots x_m] \in \Delta_m(\T)$ and
$I \in \I_m$, with $m<n$
and such that $0 \notin I$. Furthermore, let $I' = (0,I)$.
Then
\[
\int_{\Omega} \lambda_0^{-2}(x)(A_fv(P_I\lambda_f(x)) -
A_fv(P_{I'}\lambda_f(x)))^2 \, dx
\le c \|\grad v \|_{0,\Omega_f}^2, \quad v
\in H^1(\Omega_f),
\]
where the constant $c= c(\Omega,\T)$ is independent of $v$.
\end{lem}
We remark that $A_fv(P_I\lambda_f(x)) -
A_fv(P_{I'}\lambda_f(x))= 0$ outside $\Omega_{x_0}$. Therefore, the
integrand in the
integral above should be considered to be zero outside $\Omega_{x_0}$.

\section{Precise definitions and main results}\label{main}
The transform $\B = B_{\T}$ will be defined by an inductive
process which we now present.

\subsection{Definition of the transform}
We will define the map $\B$ by a recursion with respect to the
dimension of subsimplexes $f \in \Delta(\T)$.  The map $\B$ can be
defined on the space $L^2$, but the more interesting properties appear
when it is restricted to $H^1$. The main tool for constructing the
operator $\B$ are trace preserving cut--off operators $C_f$ which map
functions defined on $\Omega_f$ into functions defined on all of
$\Omega$. The operators $C_f$ are defined by utilizing the
corresponding operators $K_m$ defined on $\S_m^c$.  If $f \in
\Delta_m(\T)$, with $m <n$, then
\[
C_f v = (\lambda_f^* \circ K_m \circ A_f) v
= (K_m \circ A_f) v(\lambda_f(\cdot)).
\]
A more detailed  representation of the operator
$C_f$ is given by
\begin{equation}\label{B-def-3}
C_f v(x) = \sum_{I \in \I_m}(-1)^{|I|}\frac{\rho_f(x)}
{\rho_{f(I)}(x)}A_f v(P_I\lambda_f(x)),
 \end{equation}
where we recall that $f(I)= \{ x \in f \, | \, P_I\lambda_f(x)=
\lambda_f(x) \, \}$.  Observe that $\lambda_f \equiv (0,\ldots ,0)$
outside $\Omega_f^e$ and that all functions of the form $K_mw$ are zero
at the origin in $\R^{m+1}$. As a consequence, $\supp (C_f v)$ is
contained in the closure of $\Omega_f^e$.  For the final case when $f
\in \Delta_n(\T) = \T$, we simply define the operator $C_f$ to be the
restriction to $f$, i.e., $C_fv =v|_f$.

If $f \in \Delta_0(\T)$, i.e., $f$ is a vertex, then $B_f =
C_f$.
More generally, for each $f \in \Delta_m(\T)$
we define
\begin{equation}\label{B-def-1}
B_fu = C_f u^m, \quad \text{where }u^m
= (u - \sum_{\stackrel{g \in \Delta_j(\T)}{j<m}}B_gu).
\end{equation}
Alternatively, the functions $u^m$ satisfy $u^0=u$ and the recursion
\[
u^{m+1} = u^m - \sum_{f \in \Delta_m(\T)} C_f u^m = u^m - \sum_{f \in \Delta_m(\T)} B_f u.
\]
As a consequence of the definition of the operator $C_f$ for $\dim f = n$,
it follows by construction that $u = \sum_f B_fu$.
Furthermore, from the corresponding property of the operator $C_f$,
it also follows that $\supp (B_f u)$ is in the closure of $\Omega_f^e$.
Also, by Lemma~\ref{H1-preserve-gen}, and from the fact that $\rho_f/\rho_{f(I)} \le 1$,
it follows directly that the operator $B_f$ is bounded in $L^2$.
However, it is more challenging
to establish that $B_f$ is bounded in $H^1$, and that $B_fu \in \0 H^1(\Omega_f)$ for $u \in H^1(\Omega)$.

\subsection{Main properties of the transform}
The main arguments needed for verifying the properties (i)--(iii) of
the transform $\B$, stated in Section~\ref{prelim} above, will be
given here.  We will first establish that the piecewise polynomial
space, $\P_r\Lambda^0(\T)$, is preserved by the transform, i.e., we
will show property (iii).

\begin{thm}\label{pol-preserve-B}
If $u \in \P_r\Lambda^0(\T)$, then $B_f u \in \0\P_r\Lambda^0(\T_f)$
for all $f \in \Delta(\T)$.
\end{thm}

\begin{proof}
Assume that $u \in \P_r\Lambda^0(\T)$.
We will show that for all $m$, $0 \le m \le n$,  the
following properties hold:
\begin{equation}\label{ind-hyp-1}
u^m \in \P_r\Lambda^0(\T), \quad \text{with } \tr_g u^m = 0, \quad g \in
\Delta_j(\T), \, j < m,
\end{equation}
and
\begin{equation}\label{ind-hyp-2}
B_g u \in \0\P_r\Lambda^0(\T_g), \, g \in
\Delta_j(\T), \, j < m.
\end{equation}
Here the function $u^m$ is defined by \eqref{B-def-1}.
The proof of \eqref{ind-hyp-1} and \eqref{ind-hyp-2}
goes by induction on $m$.
Note that for $m=0$, these properties hold with $u^0 = u$.
Assume now that \eqref{ind-hyp-1} and \eqref{ind-hyp-2}
hold for a given $m$, $m
<n$. Let  $ v \equiv  u^m \in \P_r\Lambda^0(\T)$.
Then, for any $f = [x_0,x_1, \ldots x_m]\in \Delta_m(\T)$,
we have $\tr_f v \in \0\P_r(f)$.
Therefore, it follows from Lemma~\ref{pol-preserve-A}
that
\[
A_fv \in \P_r(\S_m^c) \quad \text{and }\tr_{\S_m}A_fv \in
\0\P_r(\S_m).
\]
In fact, if $\lambda \in \S_m$,
then $A_f v(\lambda) = v(x)$, where $x = \sum_{j=0}^m \lambda_jx_j
\in f$.
But from Lemma~\ref{pol-preserve-E}, we can then conclude that
\[
(K_m\circ A_f)v \in \P_r(\S_m^c), \quad \text{with } \tr_{\S_m}(I -
K_m)A_fv = 0, \,
\tr_{\partial \S_m^c\setminus \S_m}(K_m\circ A_f)v = 0.
\]
However, this implies that
\[
B_fu = C_f^m u^m = (K_m\circ
A_f)v(\lambda_f(\cdot)
\in \0\P_r\Lambda^0(\T_f),
\]
and with $\tr_f B_f u = \tr_f u^m$.
This property holds for all $f \in \Delta_m(\T)$. Therefore, since
\[
u^{m+1} = u^m - \sum_{f \in \Delta_m(\T)} B_fu,
\]
we can conclude
that \eqref{ind-hyp-1} and \eqref{ind-hyp-2} hold with
$m$ replaced by $m+1$.
This completes the induction argument. In particular, we have shown
that $B_f u \in \0\P_r\Lambda^0(\T_f)$ for
all $f \in \Delta_m(\T)$, $m < n$.
Furthermore,
$\tr_f u^n = 0$ for all
$f \in \Delta_{n-1}(\T)$. This means that
\[
u^n = \sum_{T \in \T} u_T^n, \quad u_T^n \in \0\P_r\Lambda^0(T), \, T
\in \T.
\]
Since $B_T u =  u_T^n$ for any $T \in \Delta_n(\T) = \T$,
the proof is completed.
\end{proof}

The next result will be a key step for showing properties (i) and (ii)
of the transform.

\begin{lem}\label{prop-1}
Assume that $f \in \Delta_m(\T)$, with $m <n$, and that $v \in
H^1(\Omega_f)$
with $\rho_g^{-1}v \in L^2(\Omega_f)$, where $g = f(I)$ for $I \in \I_m$.
Define $w = \frac{\rho_f}{\rho_g}
A_fv(P_I\lambda_f(\cdot))$.
Then $w \in H^1(\Omega)$ and $\rho_f^{-1}w \in L^2(\Omega)$.
\end{lem}

\begin{proof}
Since $g \in \Delta(f)$, $\rho_f/\rho_g \le 1$. Therefore,  it follows
directly from Lemma~\ref{H1-preserve-gen}
that $w \in L^2(\Omega)$. We also have from Lemma~\ref{decay-preserve}
that
\begin{align*}
\int_{\Omega}|\rho_f^{-1}w|^2 \, dx
&= \int_{\Omega}|\rho_g^{-1}A_fv(P_I\lambda_f(x))|^2 \, dx
\\
&\le c \Big[\int_{\Omega_f}|\rho_g^{-1}v(x)|^2 \, dx + \| \grad v
\|_{0,\Omega_f}^2\Big] < \infty,
\end{align*}
so the desired decay property of $w$ follows.
It remains to show that $w \in H^1(\Omega)$.
From the identity
\[
\grad (\rho_f/\rho_g) = \rho_g^{-1}(\grad \rho_f -
\frac{\rho_f}{\rho_g}\grad \rho_g),
\]
we obtain that $|\grad (\rho_f/\rho_g)| \le c_0 \rho_g^{-1}$, where
$c_0 = c_0(\Omega, \T)$.
Therefore, we can conclude that
\[
\int_{\Omega_f} |(\grad(\rho_f/\rho_g))A_fv(P_I\lambda(x))|^2 \, dx \le
c_0^2\int_{\Omega_f}|\rho_g^{-1}A_fv(P_I\lambda(x))|^2 \, dx.
\]
Together with Leibnitz' rule and the result of Lemma~\ref{H1-preserve-gen},
this will imply that $w \in H^1(\Omega)$.
This completes the proof.
\end{proof}

\begin{lem}\label{prop-2}
Let $f \in \Delta_m(\T)$ with $x_0 \in \Delta_0(f)$.
Assume that $v \in H^1(\Omega_f)$,
with the property that $\rho_g^{-1}v
\in L^2(\Omega_f)$ for all $g \in \Delta_j(f)$, $j< m$.
Then $\lambda_0^{-1} C_f v \in L^2(\Omega)$.
\end{lem}

\begin{proof}
Assume first that $m <n$.
Let $I \in \I_m$ be any index set such that $0 \notin I$.
Furthermore, let $I' = (0,I) \in \I_m$.
In other words, $x_0 \in \Delta(g)$ while $x_0 \notin \Delta(g')$,
where $g = f(I)$ and $g' = f(I')$.
The desired result will follow if we can show that
\begin{multline*}
\lambda_0^{-1} \Big[\frac{\rho_f}{\rho_g}A_fv(P_I\lambda_f(\cdot ))
- \frac{\rho_f}{\rho_{g'}} A_fv(P_{I'}\lambda_f(\cdot ))\Big]
\\
= \lambda_0^{-1} \frac{\rho_f}{\rho_g}\Big[A_fv(P_I\lambda_f(\cdot ))
-A_fv(P_{I'}\lambda_f(\cdot ))\Big]
+ \frac{\rho_f}{\rho_g\rho_{g'}}A_fv(P_{I'}\lambda_f(\cdot ))
\in L^2(\Omega).
\end{multline*}
However, Lemma~\ref{decay-cancel} and the fact that $\rho_f/\rho_g \le
1$ implies that the first term on the right hand side is in
$L^2$. Furthermore,
it follows by assumption that $\rho_{g'}^{-1}v \in L^2$, and therefore
Lemma~\ref{decay-preserve}
implies that the second term is in $L^2$.

If $m=n$, then we recall that $C_f v$ is just $v$ restricted to $f$.
If $f =[x_0,x_1, \ldots ,x_n]$ and $g = [x_1, \ldots ,x_n]$, then
$\rho_g^{-1} v = \lambda_0^{-1} v \in L^2$
by assumption. This completes the proof.
\end{proof}

\begin{lem}\label{prop-3}
Let $f = [x_0.x_1,\ldots ,x_m] \in \Delta_m(\T)$
and assume that $v \in H^1(\Omega_f)$,
with the property that $\rho_g^{-1}v
\in L^2(\Omega_f)$ for $g \in \Delta_j(f)$, $j< m$.
Define $w = C_f v$.
Then $w|_{\Omega_f} \in \0 H^1(\Omega_f)$ and $w \equiv 0$ on
$\Omega \setminus \Omega_f$.
\end{lem}

\begin{proof}
We first observe that $w|_{\Omega_f} \in H^1(\Omega_f)$. This is
obvious if $m = n$, while for $m <n$ it follows from
Lemma~\ref{prop-1} that all the terms in the series of $(K_m \circ
A_f)v(\lambda_f(\cdot))$ have this property.  To show that $w \in \0
H^1(\Omega_f)$, it is enough to show that for any vertex $x_0$ of $f$,
$w \in \0 H^1(\Omega_{x_0})$. Since the numbering of the vertices of
$f$ is arbitrary, this will in fact imply that
\[
w \in \cap_{j=0}^m \0 H^1(\Omega_{x_j}) = \0 H^1(\Omega_f).
\]
However, the property that $w \in \0 H^1(\Omega_{x_0})$
is a consequence of the decay results expressed in
Lemmas~\ref{prop-2}, i.e., that $\lambda_0^{-1}w \in L^2$.
For any $\eps > 0$, let $\phi_{\eps}$ be a smooth function on $\R$ such that
$\phi_{\eps} \equiv 0$ on $(-\eps/2,\eps/2)$, $\phi_{\eps} \equiv 1$
on the complement of $(-\eps,\eps)$, and such that
$\phi'_{\eps}(\lambda)\lambda$ is uniformly bounded, i.e.,
\begin{equation}\label{phi'-est}
|\phi'_{\eps}(\lambda)| \le c/|\lambda|, \quad
\frac{\eps}{2} \le |\lambda| \le \eps,
\end{equation}
for some constant $c$.
By construction, the functions $v_{\eps}\equiv  \phi_{\eps}(\lambda_0(\cdot)) w$
are in $\0 H^1(\Omega_{x_0})$,
and to show that $w$ belongs to the same space, it is enough to show
that the $v_{\eps} $ converge to $w$, as $\eps$ tends to zero,
in $H^1(\Omega_{x_0})$. However,
\[
\int_{\Omega_{x_0}}|v_{\eps} - w|^2 \, dx =
\int_{\Omega_{x_0}}|(\phi_{\eps}(\lambda_0(\cdot)) - 1)w|^2 \, dx \le
\int_{\Omega_{x_0,\eps}}
|w|^2 \, dx \rightarrow 0,
\]
where $\Omega_{x_0,\eps} =
\{x \in \Omega_{x_0} \, | \, \lambda_0(x) \le \eps \, \}$.
This shows the $L^2$ convergence.
Furthermore,
\[
\int_{\Omega_{x_0}}|\grad (v_{\eps} - w)|^2 \, dx
\le 2 \int_{\Omega_{x_0,\eps}}|\grad w|^2 \, dx
+ 2 \int_{\Omega_{x_0,\eps}}|(\grad (\phi_{\eps}(\lambda_0(\cdot))) w|^2 \, dx.
\]
The first term goes to zero by the $H^1$ boundedness of $w$,
and, as a consequence of \eqref{phi'-est} and the $L^2$ property of
$\lambda_0^{-1}w$ established in Lemma~\ref{prop-2}, the second
term goes to zero with $\eps$.
By completeness of $\0 H^1(\Omega_{x_0})$, it follows that $w \in \0
H^1(\Omega_{x_0})$ and therefore it is in $\0 H^1(\Omega_f)$.

We recall from the definition of the operator $C_f$ that
$w$ is identically zero on $\Omega \setminus
\Omega_f^e$. Hence, it remains to show that $w$ is identically zero
on $\Omega_f^e \setminus \Omega_f$ when  $m <n$.
However, at each point in $\Omega_f^e \setminus \Omega_f$,
at least one of the extended barycentric
coordinates associated to $f$ is zero. Therefore, $w$ in this region
corresponds to
a pullback of $w$ from $\partial \S_m^c\setminus \S_m$,
and this is zero since $\tr_{\partial \Omega_f}w = 0$.
\end{proof}

\begin{lem}\label{prop-4}
Let $u \in H^1(\Omega)$ and define the functions $u^m$, $0 \le m \le
n$, by \eqref{B-def-1}. Then $u^m \in H^1(\Omega)$
and $\rho_f^{-1} u^m \in L^2(\Omega)$
for all $f \in \Delta_j(\T)$, $j<m$.
\end{lem}

\begin{proof}
The proof goes by induction on $m$. For $m=0$ the result holds with
$u^0 = u$. Furthermore, if the result holds for a given $m <n$, then
$u^{m+1}\in H^1(\Omega)$ by Lemma~\ref{prop-3}.
It remains to show the decay property, i.e., that $\rho_f^{-1}u^{m+1}
\in L^2(\Omega)$ for all $f \in \Delta_j(\T)$ for $j \le m$.
For any $f \in \Delta_m(\T)$ we have
\begin{multline*}
\rho_f^{-1}(u^m - C_f u^m)
\\
= \rho_f^{-1}[u^m - A_fu^m(\lambda_f(\cdot ))]
 - \rho_f^{-1}\sum_{\stackrel{I \in \I_m}{I \neq \emptyset}}(-1)^{|I|}\frac{\rho_f}
{\rho_{f(I)}} A_f u^m(P_I\lambda(\cdot)).
\end{multline*}
However, the first term on the right side is in $L^2$ as a consequence of
Lemma~\ref{decay-create}, while Lemma~\ref{prop-1} and the induction
hypothesis implies that all
the terms in the sum are
in $L^2$. We can therefore conclude that for $f \in \Delta_m$,
$\rho_f^{-1}(u^m - C_f u^m)$ is in $L^2(\Omega)$.
To show that $\rho_f^{-1} u^{m+1}$ is in $L^2$, we express this as
\begin{equation}\label{representation}
\rho_f^{-1} u^{m+1} = \rho_f^{-1}(u^m - C_f u^m) + \sum_{\stackrel{g \in
    \Delta_m(\T)}{g \neq f}}\rho_f^{-1} C_g u^m.
\end{equation}
Recall that by definition, $C_g u^m$ is identically zero
outside $\Omega_g^e$. On the other hand, if $g \in
    \Delta_m(\T)$ and $g \neq f$, then on each $T \in \T$, such that
$f \cap T \neq \emptyset$ and $g \cap T \neq \emptyset$, there
exists a vertex $x_0 \in g\cap T$ which is not in $f$.
Then $\lambda_0 \le \rho_f$ on $T$, which implies that
\[
|\rho_f^{-1} C_g u^m| \le
|\lambda_0^{-1} C_g u^m|
\quad \text{on } T.
\]
By repeating this for all $T \subset \Omega_f^e$, and by
applying Lemma~\ref{prop-2},
we obtain that all the terms in the sum \eqref{representation} are in $L^2$.
Since $f \in \Delta_m(\T)$ is arbitrary, this shows the desired decay
result for all $f \in \Delta_m(\T)$. However, if $g \in \Delta(f)$, then
$\rho_g^{-1}(x) \le \rho_f^{-1}(x)$, and therefore $\rho_f^{-1}u^{m+1}
\in L^2$ for all $f \in \Delta_j(\T)$, $j\le m$.
This completes the induction argument and therefore the proof of the lemma.
\end{proof}

The following result shows that the transform satisfies properties (i)
and (ii) above.

\begin{thm}\label{local-prop-B}
Assume that $u \in H^1(\Omega)$. Then
$ u = \sum_{f \in \Delta(\T)}B_f u$, where  $B_fu \in \0
H^1(\Omega_f)$
for each $f \in \Delta(\T)$. Furthermore, the transformation
$\B_{\T} : H^1(\Omega) \to \bigoplus_{f \in \Delta(\T)}\0 H^1(\Omega_f)$,
with components $B_f$, is bounded.
\end{thm}

\begin{proof}
We have already seen that
$ u = \sum_{f \in \Delta(\T)} B_f u$.
Furthermore, it is a consequence of Lemmas~\ref{prop-3} and
\ref{prop-4} that each $B_f u \in \0 H^1(\Omega_f)$.
Finally, the boundedness of the transformation can be seen by
tracing the bounds derived in Lemmas~\ref{prop-1}--\ref{prop-4}
and by utilizing the finite overlap property of the covering
$\{\Omega_f \}$ of $\Omega$.
\end{proof}

\begin{cor}\label{L2-bounded}
The transform $\B_{\T}$ is $L^2$ bounded, with $\supp B_fu$
contained in the closure of $\Omega_f$ for all $u \in
L^2(\Omega)$.
\end{cor}

\begin{proof}
We have already seen that $\B_{\T}$ is $L^2$ bounded, and with $\supp B_fu$
contained in the closure of the extended macroelement $\Omega_f^e$. 
However, due to the result of 
Theorem~\ref{local-prop-B} and the density of $H^1(\Omega)$ in
$L^2(\Omega)$,
this implies that  $\supp B_fu$ is
contained in the closure of $\Omega_f$.
\end{proof}

\subsection{Construction of projections}\label{projections}
The result of Theorem~\ref{local-prop-B} leads immediately 
to the construction of locally defined projections into 
the finite element spaces 
$\P_r\Lambda^0(\T)$ which are uniformly bounded with respect to the
polynomial degree $r$. We just project each component $B_fu$ into the
space $\0\P_r\Lambda^0(\T_f)$ by a local projection $Q_{f,r}$. More
precisely, the locally defined global projections $\pi = \pi_{\T,r}$
will be of the form
\[
\pi u = \sum_{f \in \Delta_m(\T)}Q_{f,r} B_f u,
\]
where $Q_{f,r}$ is a local projection onto
$\0\P_r\Lambda^0(\T_f)$.
The operator $\pi$ will be a projection as a result of
Theorem~\ref{pol-preserve-B}.
If 
$Q_{f,r}$ is taken to be the local $H^1$-projection, with
corresponding operator norm equal to one, then
Theorem~\ref{local-prop-B}
implies that 
$\pi$ will be uniformly bounded in $H^1$ with respect to $r$.
On the other hand, if  $Q_{f,r}$ is taken to be the local
$L^2$-projection, then
Corollary~\ref{L2-bounded} implies uniform $L^2$ boundedness of
$\pi$
with respect to $r$.

\section{Proofs of Lemmas~\ref{H1-preserve-gen}--\ref{decay-cancel}}
\label{technical}
To complete the paper, it remains to establish
Lemmas~\ref{H1-preserve-gen}--\ref{decay-cancel},
all related to properties of the averaging operators $A_f$.
Let $f = [x_0,x_1, \ldots ,x_m] \in \Delta_m(\T)$ be as above.
Throughout this section we assume that $0 \le m < n$.
If $T \in \T_f$,  and $\lambda \in \S_m^c$, we also let
\[
A_{f,T}v(\lambda) = \aint_{T} v(G_m(\lambda,y))\, dy,
\]
such that
\[
A_fv = \sum_{T \in \T_f}\frac{|T|}{|\Omega_f|}A_{f,T}v.
\]
Before we derive more properties of the operator $A_f$
we will make some observations which will be useful below.
A simple calculation shows that for any $r \in \R$ we have
\begin{multline*}
\int_{\S_m^c} b(\lambda)^r \, d\lambda
= \int_{\S_{m-1}^c}\int_0^{b(\lambda')}(b(\lambda') - \lambda_m)^r \,
d\lambda_m \, d\lambda'
\\
= \int_{\S_{m-1}^c}\int_0^{b(\lambda')}z^r \, dz\, d\lambda'
= \int_0^{1}z^r \int_{z \le b(\lambda')} \, d\lambda' \, dz
= |\S_{m-1}^c| \int_0^{1}z^r(1 - z)^m \, dz.
\end{multline*}
Hence, we can conclude that
\begin{equation}\label{int-mu-s}
\int_{\S_m^c} b(\lambda)^r \, d\lambda < \infty, \quad
\text{for } r > -1.
\end{equation}
If $f = [x_0,x_1,\ldots x_m] \in \Delta_m(\T)$ and $T$ is an element of
$\T_{f}$, we let $f^*(T) \in \Delta_{n-m-1}(T)$
be the face opposite $f$. In other words, if $T = [x_0,x_1, \ldots
,x_n]$, then
\[
f^*(T) = [x_{m+1},\ldots , x_n] = \{ x \in T \, |\, \lambda_j(x) =
0,\, j=0,1,\ldots ,m\, \}.
\]
Any point $x \in T$ can be written uniquely as a convex combination
of $x_0, \ldots ,x_m$ and a point $q = q_f \in f^*(T)$, since
\[
x = \sum_{j=0}^n \lambda_j(x)x_j = \sum_{j=0}^m \lambda_j(x)x_j +
\rho_f(x)q_f(x),
\quad q_f(x) = \sum_{j=m+1}^n \lambda_j(x)x_j/\rho_f(x).
\]
Define $f^* = \cup_{T \in \T_{f}}f^*(T)$.
Then $f^* \subset \partial \Omega_f$, and any  $x \in \Omega_f$
can be written as
\begin{equation}\label{rep-Omega_f}
x = \sum_{j=0}^m \lambda_j(x)x_j +
\rho_f(x)q_f(x), \quad q_f(x) \in f^*.
\end{equation}
The set $f^*$ can alternatively be characterized as $f^* = \partial
\Omega_f^e \cap \partial \Omega_f$.
An illustration of the geometry of $f$, $\Omega_f$, and $f^*$
is given in
Figure~\ref{fig:macroelement} below.
In fact, if $m=n-1$, then $f^*$ consist of two vertices in
$\Delta_0(\T)$,
while if $m <n-1$, $f^*$ is a connected and piecewise flat manifold
of dimension $n-m-1$.

\begin{figure}[htb]
\setlength{\unitlength}{0.5cm}
\centering
\begin{picture}(15,15)
\multiput(0,8)(0.8,0){20}{\line(1,0){0.4}}
\put(0,8){\line(1,1){4}}
\put(0,8){\line(1,-1){6}}
\put(0,8){\line(5,3){10}}
\put(6,2){\line(1,3){4}}
\put(6,2){\line(5,3){10}}
\multiput(6,2)(-0.3,1.5){7}{\line(-1,6){0.12}}
\put(-0.85,8){$x_0$}
\put(16.2,8){$x_1$}
\put(3.5,12.5){$x_2$}
\put(9.5,14.25){$x_3$}
\put(6,1.5){$x_4$}
\put(16,8){\line(-1,1){6}}
\multiput(16,8)(-1.8,.6){7}{\line(-3,1){1.0}}
\put(4,12){\line(3,1){6}}
\put(12,7){$f$}
\put(6.5,13.25){$f^*$}
\end{picture}
\caption{The macroelement $\Omega_f \subset \R^3$, where $f$ is the line from
$x_0$ to $x_1$ and $f^*$ is the closed curve connecting $x_2,x_3,x_4$.}
\label{fig:macroelement}
\end{figure}
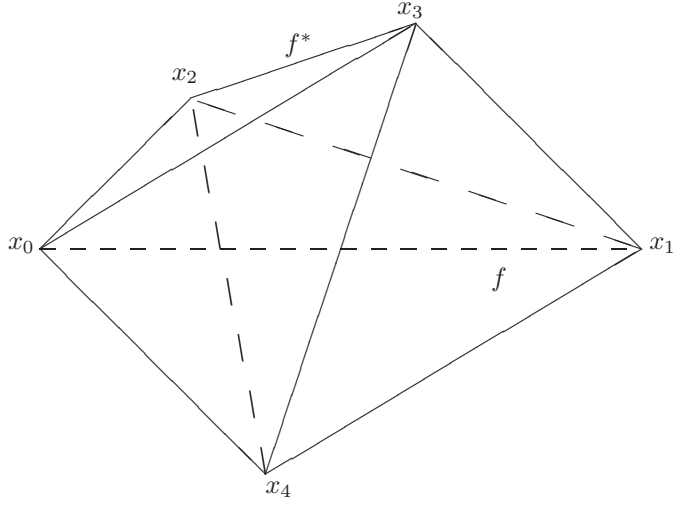

The map $x \mapsto (\lambda_f(x),q_f(x))$ defines a map from $\Omega_f$ to
$\S_m^c\times f^*$, with an inverse given by
\begin{equation}\label{polar-map}
(\lambda,q) \mapsto x
= q + \sum_{j=0}^m \lambda_j (x_j - q) = G_m(\lambda,q).
\end{equation}
The derivative of the map \eqref{polar-map} can be expressed as the
$n \times n$ matrix
\[
[x_0-q,x_1 -q, \ldots , x_m-q, b(\lambda) Q],
\]
where $Q$
is the piecewise constant  $n \times (n-m-1)$ matrix representing the
embedding of the tangent space of $f^*$
into $\R^n$. In other words, for each $T \in \T_f$ the columns of
$Q$ can be taken to be an
orthonormal basis for the tangent space of $f^*$ with respect to the
ordinary Euclidean inner product of $\R^n$.
Hence, by the scaling rule for determinants,
the determinant of this matrix is of the form
\[
b(\lambda)^{n-m-1}\det([x_0-q,x_1 -q, \ldots , x_m-q, Q])
:= b(\lambda)^{n-m-1}J(f,q).
\]
Furthermore, for a fixed mesh, the function $J(f,q)$ will be bounded
from above and below. In other words, there exist constants $c_i=
c_i(\Omega,\T)$,
such that
\begin{equation}\label{det-unif}
c_0 \le J(f,q) \le c_1, \quad f \in \Delta(\T),\, q \in f^*.
\end{equation}
The coordinates $(\lambda,q) \in \S_m^c\times f^*$ can be seen as
generalized polar coordinates for the domain $\Omega_f$.
The change of variables
\[
x \mapsto (\lambda_f(x),q_f(x)) \in \S_m^c\times f^*
\]
leads to the identity
\begin{equation}\label{polar-T}
\int_{T} \phi (\lambda_f(x),q_f(x)) \, dx =
\int_{\S_m^c} \int_{f^*(T)} \phi (\lambda,q) J(f,q)\, dq \,
b(\lambda)^{n-m-1} \, d\lambda,
\end{equation}
for any $T \in \T_f$, and any real valued function $\phi$ on $\S_m^c
\times f^*(T)$.
Furthermore, by summing over all $T \in \T_f$, we obtain
\begin{equation}\label{polar}
\int_{\Omega_f} \phi (\lambda_f(x),q_f(x)) \, dx =
\int_{\S_m^c} \int_{f^*} \phi (\lambda,q) J(f,q)\, dq \,
b(\lambda)^{n-m-1} \, d\lambda.
\end{equation}
Here the integral
over $f^*$ should be interpreted as a sum in the case $m=n-1$, when
$f^*$ consists of two points.

The function $G_m$ has the property that $G_m(\lambda_f(x),q_f(x)) = x$
and it satisfies the composition rule
\begin{equation}\label{compose}
G_m(\lambda, G_m(\mu,y)) = G_m(\lambda',y) \quad \text{where } \lambda' =
\lambda + b(\lambda)\mu.
\end{equation}
In particular, the matrix associated to the linear transformation
$\lambda \mapsto \lambda'$
is $(m+1) \times (m+1)$ given by $I - \mu e^T$, where $e$ denotes
the vector with all elements equal $1$, and this matrix has
determinant $b(\mu)$. Furthermore,
$b(\lambda') = b(\lambda)b(\mu)$.
Letting $y = G_m(\mu,q)$ and applying the identity \eqref{polar-T}
in the variable $y$, we can rewrite $A_{f,T}v(\lambda)$ as
\begin{equation}\label{A_f-polar}
A_{f,T}v(\lambda) = |T|^{-1} \int_{\S_m^c} \int_{f^*(T)} v
(G_m(\lambda,G_m(\mu,q)) J(f,q) \, dq \, b(\mu)^{n-m-1}\, d\mu,
\end{equation}
A key property, which is a special case of
Lemma~\ref{H1-preserve-gen},
is that the operator $\lambda_f^* \circ A_{f,T}$
is bounded in $L^2$. To see this, observe that we obtain from
\eqref{det-unif}, \eqref{polar}, \eqref{compose}, and Minkowski's
inequality  in the form
$\|\int g(\mu) \, d\mu \| \le \int \|g(\mu)\| \, d\mu$, that
\begin{align*}
\|&A_{f,T}v(\lambda_f(\cdot))\|_{0,\Omega_f} \\
&\le  c \int_{\S_m^c} \Big(\int_{\Omega_f} \int_{f^*(T)}|v(G(\lambda_f(x),G(\mu,q))|^2
\, dq \, dx\Big)^{1/2} b(\mu)^{n-m-1}\, d\mu\\
&\le  c \int_{\S_m^c} \Big(\int_{\S_m^c} b(\lambda)^{n-m-1}\int_{f^*(T)}
|v(G(\lambda,G(\mu,q))|^2
\, dq \, d\lambda\Big)^{1/2} {b(\mu)}^{n-m-1}\, d\mu\\
&\le c \int_{\S_m^c} \Big(\int_{\S_m^c} b(\lambda')^{n-m-1}
\int_{f^*(T)}|v(G(\lambda',q))|^2
\, dq \, d\lambda' \Big)^{1/2} {b(\mu)}^{-1 +(n-m)/2}\, d\mu,
\end{align*}
where we have substituted $\lambda' = \lambda+ b(\lambda)\mu$.
However, by letting $(\lambda',q) \mapsto x=G(\lambda',q)$,
we obtain from \eqref{polar-T} that
\begin{align*}
\|A_{f,T}v(\lambda_f(\cdot))\|_{0,\Omega_f} &\le c \int_{\S_m^c}(\int_{T} |v(x)|^2
\, dx)^{1/2} \, b(\mu)^{-1 +(n-m)/2}\, d\mu\\
&= c \|v \|_{0,T} \int_{\S_m^c} b(\mu)^{-1 +(n-m)/2}\, d\mu  \le c_1 \|v \|_{0,T},
\end{align*}
where we have used \eqref{int-mu-s} and the fact that the exponent
satisfies $-1 + (n-m)/2 \ge -1/2$.
This shows that the operator $\lambda_f^*\circ A_{f,T}$ is bounded as
an operator from $L^2(T)$ to $L^2(\Omega_f)$. Furthermore, if $T' \in
\Delta(\T)$
such that $T' \subset \Omega_f^e$, but $T' \notin \T_f$, we let $g=f
\cap T'$. Then $g \in \Delta(f)$ and $A_{f,T}v|_{T'} =
A_{g,T}v|_{T'}$.

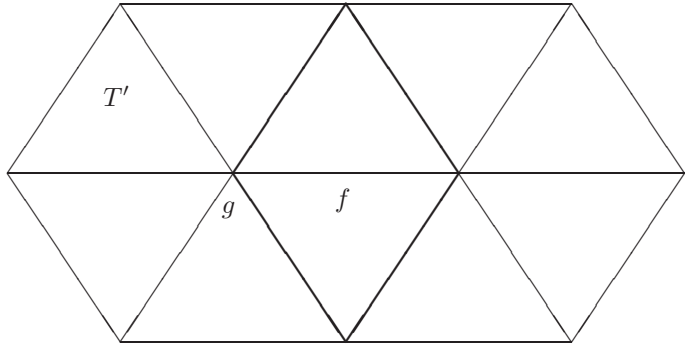
\begin{figure}[htb]
\setlength{\unitlength}{0.75cm}
\centering
\begin{picture}(4,4)
\put(0,0){\line(-2,3){2}}
\put(0,0){\line(-2,-3){2}}
\put(-4,0){\line(1,0){8}}
\put(4,0){\line(1,0){4}}
\put(4,0){\line(2,-3){2}}
\put(4,0){\line(2,3){2}}
\put(2,3){\line(1,0){4}}
\put(2,-3){\line(1,0){4}}
\put(8,0){\line(-2,3){2}}
\put(8,0){\line(-2,-3){2}}
\put(1.8,-.6){$f$}
\put(-2.3,1.2){$T'$}
\put(-.2,-.7){$g$}
\put(-4,0){\line(2,3){2}}
\put(-4,0){\line(2,-3){2}}
\put(-2,3){\line(1,0){4}}
\put(-2,-3){\line(1,0){4}}
\thicklines
\put(4,0){\line(-2,3){2}}
\put(4,0){\line(-2,-3){2}}
\put(0,0){\line(2,3){2}}
\put(0,0){\line(2,-3){2}}
\end{picture}
\vskip1in
\caption{The case when $T' \subset \Omega_f^e$, but
$T' \notin \T_f$ (enclosed in the thick lines). Here $g=f \cap T'$.}
\label{Tprime}
\end{figure}

By utilizing the argument just given with respect to $g$ instead of
$f$
we can conclude that $\lambda_f^* \circ A_{f,T}$ is bounded from
$L^2(T)$ to $L^2(\Omega_f^e)$. In particular, on the boundary of
$\Omega_f^e$,  $(\lambda_f^* \circ A_{f,T})v$ is constant with value
\[
A_{f,T}v(0) = \int_T v(y) \, dy.
\]
In fact, this is also the value of $(\lambda_f^* \circ A_{f,T})v$ in
$\Omega \setminus \Omega_f^e$, and we can therefore conclude that
$\lambda_f^* \circ A_{f,T}$ is bounded from
$L^2(T)$ to $L^2(\Omega)$. Since the operator $A_f$ is a weighted sum of the operators $A_{f,T}$,
we  can therefore conclude that $\lambda_f^* \circ A_{f}$ is bounded from
$L^2(\Omega_f)$ to $L^2(\Omega)$.

A completely analogous argument, essentially using that differentiation
commutes with averaging, also shows that $\lambda_f^* \circ A_{f}$ is
bounded from
$H^1(\Omega_f)$ to $H^1(\Omega)$.
We just observe that
\[
\grad A_{f,T}v(\lambda_f(\cdot )) = \aint_{T}  (DG_m)^T \grad v
(G_m(\lambda_f(\cdot ), y)) \, dy.
\]
Here $DG_m = DG_m(y)$ is the derivative of $G_m(\lambda_f(x),y)$ with
respect to $x$, given as the $n\times n$ matrix
\[
DG_m = \sum_{j=0}^m (x_j - y)(\grad \lambda_j)^T,
\]
and this matrix is uniformly bounded with respect to $y$.
We have therefore established Lemma~\ref{H1-preserve-gen} in the
special case when $I$ is the empty set.

\begin{proof}[Proof of Lemma~\ref{H1-preserve-gen}]
We need to show that the operators $\lambda_f^* \circ P_I^* \circ A_f$
are bounded from $L^2(\Omega_f)$ to $L^2(\Omega)$ and from
$H^1(\Omega_f)$ to $H^1(\Omega)$ for all $I \in \I_m$.
As in the discussion above, it is sufficient to consider each of the operators
$\lambda_f^* \circ P_I^* \circ A_{f,T}$ for all $T \in \T_f$.
However, the operator $\lambda_f^* \circ P_I^* \circ A_{f,T}$ is equal
to $\lambda_g^* \circ A_{g,T}$,
where $g= f(I)= \{ x \in f \, |\, P_I\lambda_f(x) = \lambda_f(x) \,
\}$, and as a consequence, the desired result follows from the
discussion above.
\end{proof}

\begin{proof}[Proof of Lemma~\ref{decay-create}]
Since the function $\rho_f$ is identically to one outside
$\Omega_f^e$ and the operator $\lambda_f^* \circ A_f$ is bounded in
$L^2$,
it is enough to show that
\[
\int_{\Omega_f^e} \rho_f^{-2}(x)|v(x) - A_fv(\lambda_f(x))|^2 \, dx
\le c \|\grad v \|_{0,\Omega_f^e}^2, \quad v
\in H^1(\Omega).
\]
Furthermore, it is enough to show the corresponding result for each of the
operators $A_{f,T}$, i.e., to show that
\begin{equation}\label{decay-create-T}
\int_{\Omega_f^e} \rho_f^{-2}(x)|v(x) - A_{f,T}v(\lambda_f(x))|^2 \,
dx \le c \|\grad v \|_{0,\Omega_f^e}^2, \quad v
\in H^1(\Omega),
\end{equation}
for all $T \in \T_f$.
In fact, it will actually be enough to show that
\begin{equation}\label{decay-create-T-}
\int_{\Omega_f} \rho_f^{-2}(x)|v(x) - A_{f,T}v(\lambda_f(x))|^2 \, dx
\le c \|\grad v \|_{0,\Omega_f}^2, \quad v
\in H^1(\Omega).
\end{equation}
For assume that \eqref{decay-create-T-}
has been established.
If $T' \in \T$, such that $T' \subset \Omega_f^e$, but $T' \notin
\T_f$, we let $g = f \cap T'$. On $T'$ we then have $\rho_f = \rho_g$,
$(\lambda_f)_i = (\lambda_g)_i$ if $x_i \in g$, and $(\lambda_f)_i = 0$
otherwise. In particular,   $A_{f,T}v = A_{g,T}v$ on $T'$.
From \eqref{decay-create-T-}, applied to $g$ instead of $f$, we then obtain
\begin{align*}
\int_{T'}\rho_f(x)^{-2}|v(x) - A_{f,T}v(\lambda_f(x))|^2 \,
dx
&\le \int_{\Omega_g}\rho_g(x)^{-2}|v(x) -
A_{g,T}v(\lambda_g(x))|^2 \, dx\\
& \le C \| \grad v \|_{0,\Omega_g}^2.
\end{align*}
By combining this with \eqref{decay-create-T-},
we obtain \eqref{decay-create-T}.

The rest of the proof is devoted to establishing the bound
\eqref{decay-create-T-}. In fact,
since smooth functions are dense in $H^1(\Omega_f)$, it is enough
to show \eqref{decay-create-T-} for such functions.
We start by introducing a new averaging operator $\tilde A_{f,T}$ by
\[
\tilde A_{f;T}v(\lambda) = \aint_{f^*(T)} v(G_m(\lambda,q)) \, dq =
\aint_{T} v(G_m(\lambda,q(y)) \, dy.
\]
In fact, if $n= m-1$ such that $f^*(T)$ is just a single vertex, then
$\tilde A_{f,T}v = v$. On the other hand, if $m < n-1$, then $f^*$ is
connected, and this is utilized below.
We will estimate the two terms
\[
\int_{\Omega_f} \rho_f^{-2}(x)|v(x) - \tilde A_{f,T}v(\lambda_f(x))|^2 \, dx,
\
\int_{\Omega_f} \rho_f^{-2}(x)|\tilde A_{f,T}v(\lambda_f(x))-
A_{f,T}v(\lambda_f(x))|^2 \, dx.
\]
Note that
\[
\tilde A_{f,T}v(0) = \aint_{f^*(T)} v(G_m(0,q)) \, dq.
\]
Since this operator reproduces constants on $f^*$, it follows by
Poincar\'{e}'s
inequality that
\begin{equation}\label{tilde-est-pre}
\int_{f^*} |v(q) - \tilde A_{f,T}v(0)|^2 \, dq \le c \| \grad v \|_{0,f^*}^2,
\end{equation}
for all functions $v \in H^1(f^*)$.
A scaling argument now shows that for any $\lambda \in \S_m^c$ we have
\begin{equation}\label{tilde-est}
\int_{f^*} |v(G_m(\lambda,q)) - \tilde A_{f,T}v(\lambda)|^2 \, dq
 \le c b(\lambda)^2
\| \grad v(G_m(\lambda,\cdot)) \|_{0,f^*}^2. \nonumber
\end{equation}
To see this, just introduce the function $\hat v $ defined on $f^*$ by
\[
\hat v(q) = v(G_m(\lambda,q)) \quad \text{with } \grad \hat v(q) =
b(\lambda)\grad v(G_m(\lambda,q)).
\]
Furthermore, $\tilde A_{f,T}\hat v(0) = \tilde A_{f,T}v(\lambda)$.
Therefore, the estimate \eqref{tilde-est} follows directly from
\eqref{tilde-est-pre}.
Furthermore, by using \eqref{polar} and \eqref{tilde-est} we obtain
\begin{align}\label{tilde-est-2}
\int_{\Omega_f}\rho_f(x)^{-2}|v(x) &- \tilde A_{f,T}v(\lambda_f(x))|^2 \, dx
\nonumber\\
&\le \int_{\S_m^c}b(\lambda)^{n-m-3}\int_{f^*} |v(G_m(\lambda,q)) -
\tilde A_{f,T}v(\lambda)|^2 J(f,q)\, dq \, d\lambda\\
&\le c \int_{\S_m^c}b(\lambda)^{n-m-1}\int_{f^*}|\grad
v(G_m(\lambda,q))|^2 \, dq \, d\lambda \nonumber\\
&\le c_1 \| \grad v \|_{0,\Omega_f}^2,\nonumber
\end{align}
for all $v$ such that $v(G_m(\lambda, \cdot))$ is in $H^1(f^*)$ for
all $\lambda \in \S_m^c$. In particular, this estimate holds if $v
\in H^1(\Omega_f)$ is smooth, and this is the desired estimate
for $v - \tilde A_{f,T}v$.

To complete the proof, we need a corresponding estimate for
$\tilde A_{f,T}v(\lambda_f(\cdot )) - A_{f,T}v(\lambda_f(\cdot ))$.
For any $\lambda \in \S_m^c$ we have
\begin{align*}
\tilde A_{f,T}v(\lambda) &- A_{f,T}v(\lambda) =
- \aint_{T} [ v(G_m(\lambda,q_f(y)) - v(G_m(\lambda,y))]\, dy\\
&= b(\lambda)\aint_{T} \int_0^1 \grad
v(G_m(\lambda,(1-t)q_f(y)+ t y))\cdot (y-
q(y)) \, dt \, dy.
\end{align*}
However, writing
\begin{equation*}
y = \sum_{j=0}^m \lambda_j(y) x_j + \rho_f(y) q_f(y),
\end{equation*}
it is easy to check that
\[
G_m(\lambda,(1-t)q_f(y)+ t y) = G_m(\lambda', q_f(y)),
\]
where $\lambda' = \lambda'(\lambda, t, \lambda_f(y))$ and
\[
\lambda'(\lambda,t,\mu) = \lambda + t
b(\lambda)\mu, \quad \lambda, \mu \in \S_m^c, \, t \in \R.
\]
Therefore, since $y = G_m(\lambda_f(y), q_f(y))$, we
can use \eqref{polar-T} to rewrite the representation
of $\tilde A_{f,T}v(\lambda) - A_{f}v(\lambda)$ in the form
\begin{multline*}
\tilde A_{f,T}v(\lambda) - A_{f}v(\lambda)
= \frac{b(\lambda)}{|T|}
\\
\cdot \int_0^1 \int_{\S_m^c} b(\mu)^{n-m-1}\int_{f^*(T)}
\grad v(G_m(\lambda'(\lambda, t, \mu), q)) \cdot (y-
q) J(f,q)\, dq \, d\mu \, dt,
\end{multline*}
where $\mu = \lambda_f(y)$ and $q = q_f(y)$.
Hence, it follows by Minkowski's inequality
and \eqref{polar-T} that
\begin{multline*}
\Big(\int_{\Omega_f} \rho_f^{-2}(x)(\tilde A_{f,T}v(\lambda(x)-
A_{f,T}v(\lambda(x))^2 \, dx\Big)^{1/2}\\
\le c \int_0^1\int_{\S_m^c}{b(\mu)}^{n-m-1}
\Big( \int_{\Omega_f} \int_{f^*} |\grad v(G_m(\lambda'(\lambda_f(x),t, \mu),q))|^2
dq \, dx\Big)^{1/2} d\mu \, dt \\
 \le c \int_0^1 \int_{\S_m^c}{b(\mu)}^{n-m-1} \Big( \int_{\S_m^c}
{b(\lambda)}^{n-m-1} \int_{f^*}
|\grad v(G_m(\lambda',q))|^2
dq \, d\lambda\Big)^{1/2} d\mu \, dt,
\end{multline*}
where $\lambda' = \lambda'(\lambda,t, \mu)$.
To proceed, we make the substitution $\lambda \mapsto \lambda'$.
The matrix associated to this transformation is
$I - t \mu e^T$, with determinant $b(t \mu)$. Here, as
above, $e$ is the vector with all components equal to one.
Furthermore, $b(\lambda') = b(\lambda)b(t\mu)$.
Since $b(t\mu) \ge b(\mu)$, it follows, again using
\eqref{polar}, that
\begin{multline*}
\Big (\int_{\Omega_f} \rho_f^{-2}(x)(\tilde A_{f,T}v(\lambda_f(x)-
A_{f,T}v(\lambda_f(x))^2 \, dx\Big)^{1/2}\\
\le c \int_0^1\int_{\S_m^c}\frac{{b(\mu)}^{n-m-1}}{{b(t\mu)}^{(n-m)/2}}
\Big( \int_{\S_m^c}
{b(\lambda')}^{n-m-1}\int_{f^*} \hskip-1.5pt
|\grad v(G_m(\lambda',q))|^2
dq \, d\lambda'\Big)^{1/2} d\mu \, dt\\
\le c \int_{\S_m^c}{b(\mu)}^{-1 + (n-m)/2}\Big( \int_{\S_m^c}
{b(\lambda')}^{n-m-1}\int_{f^*}
|\grad v(G_m(\lambda',q))|^2
dq \, d\lambda'\Big)^{1/2} d\mu \\
\le c \| \grad v \|_{0,\Omega_f}\int_{\S_m^c}{b(\mu)}^{-1 +(n-m)/2} \, d\mu
\le c \| \grad v \|_{0,\Omega_f}.
\end{multline*}
Together with \eqref{tilde-est-2}, this completes the proof of
\eqref{decay-create-T-} and hence the lemma is established.
\end{proof}

\begin{proof}[Proof of Lemma~\ref{decay-preserve}]
For $f \in \Delta_m(\T)$ and $I \in \I_m$, with $m<n$, we have to show
\[
\int_{\Omega}\rho_g^{-2}(x) |A_fv(P_I\lambda_f(x))|^2 \, dx \le c\,
[\int_{\Omega}\rho_g^{-2}(x) |v(x)|^2 \, dx + \| \grad v \|_{0}^2],
\]
where $g = f(I) \in \Delta(f)$.
We observe that
\[
A_fv(P_I\lambda_f) = \sum_{T \in \T_f} \frac{|T|}{|\Omega_f|}A_{g,T}(\lambda_g).
\]
However, by \eqref{decay-create-T}    we have
\[
\int_{\Omega} \rho_g^{-2}(x)|v(x) - A_{g,T} v(\lambda_g(x))|^2\, dx
\le c\, \| v \|_1^2,
\]
and by the triangle inequality this  implies that
\[
\int_{\Omega}\rho_g^{-2}(x) |A_{g,T}v(\lambda_g(x))|^2 \, dx \le c\,
[\int_{\Omega}\rho_g^{-2}(x) |v(x)|^2 \, dx + \| \grad v \|_{0}^2].
\]
The desired result follows by summing over $T \in \T_f$.
\end{proof}

\begin{proof}[Proof of Lemma~\ref{decay-cancel}]
Let $m<n$, $f= [x_0,x_1, \ldots x_m] \in \Delta_m(\T)$, $I \in \I_m$
with $0 \notin I$
and $I' = (0,I)$.
We must show that
\[
\int_{\Omega_{x_0}} \lambda_0^{-2}(x)(A_fv(P_I\lambda_f(x)) -
A_fv(P_{I'}\lambda_f(x)))^2 \, dx
\le c \|\grad v \|_{0,\Omega_f}^2, \quad v
\in H^1(\Omega_f).
\]
We recall that for any $T \in \T_f$ we have
$A_{f,T}v(P_I\lambda_f(\cdot)) = A_{g,T}v(\lambda_g(\cdot))$,
where $g = f(I) \in \Delta(f)$. Similarly,
$A_{f,T}v(P_I'\lambda_f(\cdot)) = A_{g,T}v(P\lambda_g(\cdot))$,
where $(P\lambda_g)_0 = 0$, and $(P\lambda_g)_i = (\lambda_g)_i$ for
$i \neq 0$.
The desired estimate will follow if we can show
\begin{equation}\label{cancel}
\int_{\Omega_{x_0}} \lambda_0^{-2}(x)(A_{g,T}v(\lambda_g(x)) -
A_{g,T}v(P\lambda_{g}(x)))^2 \, dx
\le c \|\grad v \|_{0,T}^2,
\end{equation}
for all $v \in H^1(T)$, $T \in \T_f$. In fact, it is enough to show
that
\begin{equation}\label{cancel-g}
\int_{\Omega_{g}} \lambda_0^{-2}(x)(A_{g,T}v(\lambda_g(x)) -
A_{g,T}v(P\lambda_{g}(x)))^2 \, dx
\le c \|\grad v \|_{0,T}^2.
\end{equation}
To see this, assume that $\hat T \in \T_{x_0}$ such that $\hat T \notin
\T_g$. Let $\hat g = g \cap \hat T$.
Then $\hat T \in \T_{\hat g}$, and $(\lambda_{\hat g})_i =
(\lambda_g)_i$
for all the components of $\lambda_g$ which are not identically
 zero on $\hat T$. Therefore \eqref{cancel-g}, applied
to $\hat g$ instead of $g$, will imply that
\[
\int_{\hat T} \lambda_0^{-2}(x)(A_{g,T}v(\lambda_g(x)) -
A_{g,T}v(P\lambda_{g}(x)))^2 \, dx
\le c \|\grad v \|_{0,T}^2.
\]
By carrying out this process for all possible $\hat T \in \Omega_{x_0} \setminus
\Omega_g$ and combining it with  \eqref{cancel-g}, we obtain
\eqref{cancel}.

The rest of the proof is devoted to establish \eqref{cancel-g}.
Without loss of generality we can assume that $g = [x_0, x_1, \ldots
,x_j]$ such that
\[
A_{g,T}v(P\lambda_{g}) = \aint_{T} v(G_j(\lambda_g,y) + \lambda_0(y-x_0))  \,
dy.
\]
We have
\begin{align*}
A_{g,T}v(P\lambda_{g}) - A_{g,T}v(\lambda_g)
&= \aint_{T} [v(G_j(\lambda,y) + \lambda_0(y-x_0)) - v(G_j(\lambda,y))] \,
dy\\
&= \lambda_0 \aint_{T} \int_0^1 \grad v(G_j(\lambda,y) + t\lambda_0(y-
x_0))\cdot (y- x_0) \, dt \, dy,
\end{align*}
where $\lambda = \lambda_g \in \S_j^c$. If we express
$y$ as $y = G_j(\mu,q)$, where
$\mu= \lambda_g(y)$ and $q = q_g(y)$, we further obtain that
\begin{align*}
G_j(\lambda,y) + t\lambda_0(y-x_0)
&= \sum_{i=0}^j\lambda_ix_i + (t\lambda_0 + b(\lambda))y - t\lambda_0x_0\\
&= \sum_{i=0}^j\lambda_ix_i + (t\lambda_0 +b(\lambda))
(\sum_{i=0}^j\mu_ix_i + b(\mu)q) - t\lambda_0x_0\\
&= \sum_{i=0}^j\lambda'_ix_i + b(\lambda')q = G_j(\lambda',q),
\end{align*}
where $\lambda' = \lambda'(\lambda,t,\mu)$ is given by
\[
\lambda'_0 = (1-t)\lambda_0 + (t\lambda_0 +b(\lambda))\mu_0
\]
and where
\[
\lambda'_i = \lambda_i + (t\lambda_0 + b(\lambda))\mu_i, \quad i>0.
\]
Using the identity \eqref{polar-T}, we therefore have
\begin{multline*}
A_{g,T}v(P\lambda_{g}) - A_{g,T}v(\lambda_g)\\
= \frac{\lambda_0}{|T|}
\int_{\S_j^c}b(\mu)^{n-j-1}
\int_0^1 \int_{g^*(T)}\grad v(G_j(\lambda',q))
\cdot (G_j(\mu,q)- x_0) \, dq \, dt \, d\mu,\\
\end{multline*}
where $\lambda' = \lambda'(\lambda,t,\mu)$ and $\lambda = \lambda_g$.
The matrix associated to the linear transformation
$\lambda \mapsto \lambda'$ is given by
\[
I - \mu e^T  + t(\mu- e_0)e_0^T = (I - \mu e^T)(I -
te_0e_0^T),
\]
with determinant $(1-t)b(\mu)$.

From Minkowski's inequality and \eqref{polar-T} we now have
\begin{multline*}
\Big(\int_{\Omega_g} \lambda_0^{-2}(x)|A_{g,T}v(P\lambda_{g}(x)) - A_{g,T}v
(\lambda_g(x))|^2
 \, dx\Big)^{1/2} \\
\le c \int_{\S_j^c}b(\mu)^{n-j-1}
\int_0^1 \Big(\int_{\Omega_f}\int_{g^*(T)}|\grad
v(G_j(\lambda'(x),q))|^2 \,
dq \, dx \Big)^{1/2}dt \, d\mu\\
\le c \int_{\S_j^c}{b(\mu)}^{n-j-1}
\int_0^1 \Big(\int_{\S_j^c}{b(\lambda)}^{n-j-1} \int_{g^*(T)} \hskip-12pt
|\grad v(G_j(\lambda',q))|^2 \, dq  \, d\lambda\Big)^{1/2}dt \, d\mu,
\end{multline*}
where $\lambda' = \lambda'(\lambda,t,\mu)$ is given above, and 
$\lambda'(x)= \lambda'(\lambda_g(x),t,\mu)$.
To proceed we make the substitution $\lambda \mapsto \lambda'$.
We note
\[
b(\lambda') = b(\lambda)b(\mu) + t\lambda_0 b(\mu)
\ge b(\lambda) b(\mu), 
\]
and that $\lambda$ can be regarded as function of $\lambda',t$ and $\mu$.
Therefore, we obtain
\begin{multline*}
\Big(\int_{\Omega_g} \lambda_0^{-2}(x)|A_{g,T}v(P\lambda_{g}(x))
- A_{g,T}v(\lambda_g(x))|^2 \, dx\Big)^{1/2} \\
\le  c \int_{\S_j^c} \int_0^1 \frac{{b(\mu)}^{n-j-3/2}}{(1-t)^{1/2}}
\Big(\int_{\S_j^c}{b(\lambda)}^{n-j-1} \int_{g^*(T)}\hskip -14pt |\grad
v(G_j(\lambda',q))|^2 \, dq
\, d\lambda'\Big)^{1/2}dt \, d\mu\\
\le c \int_{\S_j^c} \int_0^1 \frac{{b(\mu)}^{-1+ (n-j)/2}}{(1-t)^{1/2}}
\Big(\int_{\S_j^c} \hskip-8pt {b(\lambda')}^{n-j-1} \hskip -4pt
\int_{g^*(T)} \hskip -16pt
|\grad v(G_j(\lambda',q))|^2 \, dq
 \, d\lambda'\Big)^{1/2}dt \, d\mu\\
\le c \Big(\int_T | \grad v(x)|^2 \, dx \Big)^{1/2},
\end{multline*}
where \eqref{polar-T} has been used for the final inequality,
and where the integrals in $\mu$ and $t$ are easily seen to be
bounded.
This completes the proof of \eqref{cancel-g}, and hence of the lemma.
\end{proof}

\bibliographystyle{amsplain}

\bibliography{bubble-I}

\end{document}